\documentclass[a4paper]{amsart}
\usepackage{graphicx,amsmath}
\usepackage{amssymb,mathrsfs}
\usepackage{amsthm}
\usepackage{a4wide}

\newtheorem{theorem}{Theorem}[section]
\newtheorem{lemma}[theorem]{Lemma}
\newtheorem{proposition}[theorem]{Proposition}

\theoremstyle{definition}

\theoremstyle{remark}\newtheorem{Rem}{Remark}

\numberwithin{equation}{section}

\begin{document}

\title[On Kurzweil's 0-1 Law]{On Kurzweil's 0-1 Law in Inhomogeneous Diophantine Approximation}
\author{Michael Fuchs}
\address{Department of Applied Mathematics, National Chiao Tung University, Hsinchu 300, Taiwan}
\email{mfuchs@math.nctu.edu.tw}

\author{Dong Han KIM}
\address{Department of Mathematics Education, Dongguk University -- Seoul, Seoul 100-715, Korea}
\email{kim2010@dongguk.edu}

%\thanks{}

\begin{abstract}
We give a sufficient and necessary condition such that for almost all $s\in{\mathbb R}$
\[
\|n\theta-s\|<\psi(n)\qquad\text{for infinitely many}\ n\in{\mathbb N},
\]
where $\theta$ is fixed and $\psi(n)$ is a positive, non-increasing sequence.
This improves upon an old result of Kurzweil and contains several previous results as special cases: two theorems of Kurzweil, a theorem of Tseng and a recent result of the second author.
Moreover, we also discuss an analogue of our result in the field of formal Laurent series which has similar consequences.
\end{abstract}


\maketitle

\section{Introduction and Results}

This paper is concerned with metric inhomogeneous Diophantine approximation. More precisely, we consider the inhomogeneous Diophantine approximation problem
\begin{equation}\label{idap}
\|n\theta-s\|<\psi(n)
\end{equation}
whose number of solutions in $n\in{\mathbb N}$ is sought. Here and throughout this paper, $\theta,s\in{\mathbb R}$, $\|\cdot\|$ denotes the distance to the nearest integer and $\psi(n)$ is a (fixed) positive, non-increasing sequence which is called {\it approximation sequence}. In addition, we will sometimes assume that $\psi(n)$ is a {\it Khintchine sequence} which means that $n\psi(n)$ is non-increasing.

There are two different ways of looking at (\ref{idap}): (i) $s$ is fixed and one is interested in the number of solutions for almost all $\theta$ (with respect to the Lebesgue measure), or (ii) $\theta$ is fixed and one is interested in the number of solutions for almost all $s$. Alternatively, one can also consider the number of solutions for almost all $(\theta,s)$ (with respect to the two-dimensional Lebesgue measure). However, we will not consider this ``double-metric" situation in this paper.

First, we recall what is known for the first case where $s$ is fixed. Here, it was proved by Khintchine \cite{Khintchine} for Khintchine sequences and $s=0$ (homogeneous Diophantine approximation) that (\ref{idap}) has either finitely many solutions in $n\in{\mathbb N}$ for almost all $\theta$ or infinitely many solutions in $n\in{\mathbb N}$ for almost all $\theta$ with the latter happening if and only if
\[
\sum_{n\geq 1}\psi(n)=\infty.
\]
This result was extended to general $s$ (inhomogeneous Diophantine approximation) by Sz\"{u}sz \cite{Szusz}. Another extension was given by Schmidt \cite{Schmidt} whose (very general) result in particular implies that the previous results of Khintchine and Sz\"{u}sz hold for all non-increasing approximation sequences. This line of research was then extended in many different directions; see the monograph \cite{Harman}.

The second case where $\theta$ is fixed was considerably less studied.
Here, in his pioneering work, Kurzweil \cite{Kurzweil} showed that also a $0$-$1$ law holds.
\begin{theorem}[Kurzweil's $0$-$1$ Law \cite{Kurzweil}]
Let $\psi(n)$ be a positive, non-increasing sequence and $\theta$ be an irrational number. Then,
\begin{equation}\label{inf-sol}
\|n\theta-s\|<\psi(n)\qquad\text{for infinitely many}\ n\in{\mathbb N}
\end{equation}
either for almost all $s$ or for almost no $s$.
\end{theorem}
It is an immediate consequence of the lemma of Borel-Cantelli that $\sum_{n\geq 1}\psi(n)=\infty$ is a necessary condition for (\ref{inf-sol}). Thus, it is natural to ask when this is also sufficient. The answer to this question was also given by Kurzweil in \cite{Kurzweil} where he showed that the above condition is sufficient and necessary exactly for the set of badly approximable $\theta$ (Kurzweil's theorem). It seems that his paper was forgotten for a long time. However, in recent years, there was a revival of interest in his study with many follow-up papers; see for instance \cite{Fayad}, \cite{Kim}, \cite{Tseng}.

The main goal of this paper is to give a necessary and sufficient condition for (\ref{inf-sol}) which holds for all $\theta$ (not only badly approximable $\theta$). For Khintchine sequences such a result was already proved in \cite{Kim}. Here, we will extend the result from \cite{Kim} to all positive, non-increasing sequences $\psi(n)$. Not surprisingly, our condition will involve Diophantine approximation properties of $\theta$. More precisely, our result reads as follows.

\begin{theorem}\label{main-thm}
Let $\psi(n)$ be a positive, non-increasing sequence and $\theta$ be an irrational number with principal convergents $p_k/q_k$. Then, for almost all $s \in \mathbb R$,
\[
\|n\theta-s\|<\psi(n)\quad\text{for infinitely many}\ n\in{\mathbb N}
\]
if and only if
\begin{equation}\label{nec-suf-cond}
\sum_{k=0}^{\infty}\left(\sum_{n=q_k}^{q_{k+1}-1}\min(\psi(n),\| q_k\theta\|)\right)=\infty.
\end{equation}
\end{theorem}

This result contains several previous results as special cases: the above mentioned theorem of Kurzweil as well as its extensions given in his paper \cite{Kurzweil} and by Tseng \cite{Tseng}. We recall these results and show that our result implies them below.

Analogues of all the above results were also obtained in the field of formal Laurent series over a finite base field; see \cite{ChenFuchs}, \cite{KimNakada}, \cite{KiNaNa}, \cite{KiTaWaXu}, \cite{Lin}. Indeed, an analogue of our Theorem \ref{main-thm} also holds in this situation, again implying many previous results. This will be shown below as well.

We conclude the introduction by giving a short sketch of the paper. In the next section, we will prove our main result. In Section \ref{Ku-Ts}, we will show that our result implies the previous results of Tseng and Kurzweil (which will also be recalled in this section). In Section \ref{Kim}, we will consider Khintchine sequences $\psi(n)$ and show that in this case (\ref{nec-suf-cond}) is equivalent to the condition from the main result of \cite{Kim}. Finally, in Section \ref{formal-Laur}, we will discuss an analogue of our result in the field of formal Laurent series over a finite field (whose definition will be recalled in this section) and show relations of this analogue to previous results.

\section{Proof of The Main Theorem}

We first fix some notation.  Let $X = \mathbb R/\mathbb Z$. Let $B(x,r)$ be the open ball in $X$ centered at $x$ with radius $r$. We denote by $\mu$ the Lebesgue measure on the unit circle $X$. Let $\psi(n)$ be a positive, non-increasing sequence and $\theta$ be an irrational with principal convergents $p_k/q_k$.

We first consider the convergence part of Theorem \ref{main-thm}.

\subsection{Proof of the Convergence Part.} The result follows from the following lemma.
\begin{lemma}\label{main-thm-conv}
If
\[
\sum_{k=0}^\infty \left(\sum_{n=q_k}^{q_{k+1}-1}\min(\psi(n),\| q_k\theta\|)\right)<\infty,
\]
then, for almost all $s\in{\mathbb R}$,
\[
\|n\theta-s\|<\psi(n)\quad\text{for finitely many}\ n \in{\mathbb N}.
\]
\end{lemma}

\begin{proof} In the proof (and also below), we will use the following well-known facts about the sequence $q_k$:
\[
\frac{1}{2}\leq q_{k+1}\|q_k\theta\|\leq 1\qquad\text{and}\qquad q_{k+1}\geq 2q_{k-1}.
\]
We will consider two cases.

In the first case, we assume that $\psi(q_{k+1}-1)\geq\|q_k\theta\|$ for infinitely many $k$. Then, for such $k$, we have
\[
\psi(n)\geq\psi(q_{k+1}-1)\geq\|q_k\theta\|
\]
for all $q_{k-1}\leq n<q_{k+1}$. Hence,
\[
\sum_{n=q_{k-1}}^{q_k-1}\min(\psi(n),\|q_{k-1}\theta\|)+\sum_{n=q_k}^{q_{k+1}-1}\min(\psi(n),\|q_k\theta\|)
\geq(q_{k+1}-q_{k-1})\|q_k\theta\|\geq\frac{q_{k+1}}{2}\|q_k\theta\|\geq\frac{1}{4}.
\]
Since this happens infinitely often, the convergence assumption is violated and thus this case will not happen.

Therefore, we may assume that $\psi(q_k-1)<\|q_{k-1}\theta\|$ for all $k$ large enough. In order to prove our claim in this case, set
$$ E_{k+1} = \bigcup_{q_k \le n < q_{k+1}} B( n\theta, \psi(n) ).$$
Then, we have
$$\bigcap_{N \ge 1} \bigcup_{n \ge N} B( n\theta, \psi(n) )
=\bigcap_{K \ge 1} \bigcup_{k \ge K} E_k.$$
Since $\| n \theta - (n-q_k) \theta\| = \|q_k \theta \|$ and $\psi(n)$ is non-increasing, we have for each $q_{k} \le n < q_{k+1}$
$$
\mu \Big( B( n \theta, \psi(n) ) \setminus B( (n-q_k)\theta, \psi(n-q_k) ) \Big) \le \| q_k \theta \|.
$$
We also have that for each $q_k \le n < q_{k+1}$
$$
\mu \Big( B( n \theta, \psi(n) ) \setminus B( (n-q_k)\theta, \psi(n-q_k) ) \Big) \le \mu \left( B( n \theta, \psi(n) ) \right) = 2 \psi (n).
$$
Thus,
\begin{equation*}
\begin{split}
\mu(E_{k+1}) &\le \sum_{n= q_k}^{2q_{k}-1} \mu \left( B(n\theta, \psi(n)) \right)
 + \sum_{n= 2q_k}^{q_{k+1}-1} \mu \Big( B(n\theta, \psi(n)) \setminus B((n-q_k)\theta, \psi(n-q_k)) \Big) \\
%&\le \sum_{n= q_k}^{2q_{k}-1} 2\psi(n) + \sum_{n= 2q_k}^{q_{k+1}-1} \min ( 2\psi(n), \| q_k \theta \|) \\
&\le 2 q_k \psi(q_k) + \sum_{n= 2q_k}^{q_{k+1}-1} \min (2\psi(n), \| q_k \theta \|).
\end{split}
\end{equation*}
Now, from $\psi( q_k -1)<\| q_{k-1}\theta\|$,
\begin{equation*}
\begin{split}
q_k \psi(q_k) &\leq 2(q_k - q_{k-2}) \psi(q_{k}-1) = 2(q_k - q_{k-2}) \min ( \psi(q_{k}-1), \| q_{k-1} \theta\|) \\
&\le 2 \left( \sum_{n=q_{k-2}}^{q_{k-1}-1} \min ( \psi(n), \| q_{k-2} \theta\| ) + \sum_{n=q_{k-1}}^{q_{k}-1} \min (\psi(n), \|q_{k-1} \theta\|)  \right).
\end{split}
\end{equation*}
Since this holds for all large $k$, we have
\[
\sum_{k}\mu(E_{k+1})<\infty.
\]
Hence, by the first Borel-Cantelli lemma, we complete the proof.
\end{proof}

\subsection{Proof of the Divergence Part.} Now, we prove the second half of Theorem \ref{main-thm}.
First, for each $n \in \mathbb N$ denote by $h(n)$  the non-increasing sequence
$$ h(n) := \min( \psi(n), \| q_k \theta \| ), \quad q_k \le n < q_{k+1}.$$
Let, for $0 \le i < a_{k+1}$,
$$
G_{k,i} := \bigcup_{q_{k+1} - (i+1)q_{k} < n \le q_{k+1} - iq_{k}} B\left( n\theta, \frac{h(q_{k+1} -iq_k)}{2} \right)
$$
and
$$ G_k := \bigcup_{i = 0}^{a_{k+1}-1} G_{k,i}.$$

Then, balls in $G_{k}$ are disjoint since
%$$  \psi(q_{k+1} - i q_k) \le  \psi(q_{k+1} - b_{k+1} q_k) \le \psi(q^*_k) \le \| q_k \theta \|$$
any two points in $ \{ n\theta : 1 \le n \le q_{k+1} \} $ are separated by at least $\| q_k \theta \|$.

%\[ \frac{q_{k+1} - q^*_k}{q_{k}}  < b_{k+1} \le \frac{q_{k+1} - q^*_k}{q_{k}}  \]

\begin{lemma}\label{thm3}
If $$ \sum_{k=0}^\infty \mu (G_k) = \infty,$$
then
$$ \mu \left(  \bigcap_{K \ge 1} \bigcup_{k \ge K} G_k  \right) = 1. $$
\end{lemma}

\begin{proof}
We estimate $\mu(G_\ell \cap G_k)$, $\ell< k$ by the Denjoy-Koksma inequality (see, e.g., \cite{Herman}): let $T$ be an irrational rotation by $\theta$ and $f$ be a real valued function of bounded variation on the unit interval. Then, for any $x$, we have
\begin{equation}\label{Koksma}
\left| \sum_{n=0}^{q_k -1} f(T^n x) - q_k \int f \, d\mu \right | \leq \text{\rm var} (f). \end{equation}

For a given interval $I$, by the Denjoy-Koksma inequality (\ref{Koksma}), we have
\begin{equation*}
\# \left\{ 0 \le n < q_{k} : n\theta \in I \right\} = \sum_{n=0}^{q_{k}-1} 1_{I} (T^n x) \leq q_{k} \mu (I) + 2.
\end{equation*}
Since $G_{k,i}$ consists of the disjoint balls centered at $q_k$ orbital points %of $(q_{k+1} - (i+1) q_k)\theta$
 with radius $r:= h (q_{k+1} - iq_{k}) /2$, we have for each $i$
$$ \mu \left( G_{k,i} \cap I \right)
\leq  \# \left\{ 0 \le n < q_{k} : n\theta \in I \right\} \cdot 2r + 2r
\leq \left( q_{k} \mu(I) + 3 \right) \cdot 2r = \mu(G_{k,i}) \mu(I) + \frac{3}{q_k} \mu (G_{k,i}).$$

Note that $G_\ell$ consists of at most $q_{\ell+1}$ intervals.

Therefore, we have for $k > \ell$
\begin{equation*}
\mu( G_{k,i} \cap G_\ell )
\leq \mu( G_{k,i} ) \mu( G_\ell ) + \frac{3 q_{\ell+1}}{q_k} \mu(G_{k,i}).
\end{equation*}

Since $G_k = \cup G_{k,i}$ is a disjoint union, we have
\begin{equation*}\begin{split}
\mu( G_{k} \cap G_\ell ) &\leq \mu( G_{k} ) \mu( G_\ell ) + \frac{3 q_{\ell+1}}{q_k} \mu(G_{k}) \\
&\leq \mu( G_{k} ) \mu( G_\ell ) + 3 \left(\frac{1}{2}\right)^{\lfloor (k-\ell-1)/2 \rfloor} \mu(G_{k}) \\
&\le \mu( G_{k} ) \mu( G_\ell ) + \frac{6}{2^{(k-\ell)/2}} \mu(G_{k}).
\end{split}
\end{equation*}

We need a version of Borel-Cantelli lemma (e.g. \cite {Harman, Sp}) to go further:
\begin{lemma}\label{Sp}
Let $(\Omega, \mu)$ be a measure space,
let $f_k(\omega)$ $(k=1,2,\dots)$ be a sequence of nonnegative $\mu$-measurable functions, and let $\varphi_k$ be a sequence of real numbers such that
$ 0 \le \varphi_k \le1$ $(k = 1, 2, \dots).$
Suppose that
$$ \int_\Omega \left( \sum_{m < k \le n} f_k (\omega) - \sum_{m < k \le n} \varphi_k \right)^2 d\mu \le C \sum_{m < k \le n} \varphi_k$$
for arbitrary integers $m<n$. Then,
$$ \sum_{1 \le k \le n } f_k (\omega) = \Phi(n)  + {\mathcal O}( \Phi^{1/2}(n) \ln^{3/2+\varepsilon} \Phi (n))$$
for almost all $\omega \in \Omega$, where $\varepsilon >0$ is arbitrary and $\Phi(n) = \sum_{1 \le k \le n} \varphi_k$.
\end{lemma}

Put $\varphi_k = \mu(G_k)$ and $f_k (\omega) = 1_{G_k} (\omega)$ in Lemma~\ref{Sp}.
Then, we have, for any $m<n$,
\begin{equation*}
\begin{split}
&\int \left( \sum_{m < k \le n} f_k (\omega) - \sum_{m < k \le n} \varphi_k \right)^2 d\mu \\
%&= \sum_{m < k \le n}  \sum_{m < \ell \le n} \left ( \int 1_{G_k} 1_{G_\ell} d\mu - \mu (G_k) \mu (G_\ell) \right ) \\
&\le 2 \sum_{m < \ell < k \le n} \left( \mu ( G_k \cap G_{\ell} ) - \mu (G_k) \mu (G_\ell) \right)  +  \sum_{m < k \le n} \mu(G_k) \\
&\leq 2 \sum_{m < k \le n}  \sum_{m < \ell < k} \frac{6}{2^{(k-\ell)/2}} \mu (G_k) + \sum_{m < k \le n} \mu(G_k)
\leq \left( \frac{12}{\sqrt{2} -1} +1 \right) \sum_{m < k \le n} \mu (G_k).
\end{split}
\end{equation*}
Therefore, by Lemma~\ref{Sp},
if
$$ \sum_k \mu (G_k) = \infty,$$
then, we have for almost every $\omega$,
$$\sum_{k=1}^\infty 1_{G_k} (\omega) = \infty,$$
i.e.,
\begin{equation*} \omega \in G_k \quad \text{ for  infinitely many $k$'s}. \qedhere\end{equation*}
\end{proof}

\begin{lemma}
If $\sum_{n = 1}^{\infty} h (n) = \infty$, then we have
$$\sum_{k=0}^\infty \mu (G_k) = \infty. $$
\end{lemma}

\begin{proof}
Since for $k \ge 0$
\begin{align*}
\sum_{n = q_k}^{q_{k+1} -1} h(n)
&= \sum_{n=q_k}^{q_k+q_{k-1}-1} h(n) + \sum_{i=1}^{a_{k+1}-1} \left( \sum_{n = iq_k + q_{k-1}}^{(i+1)q_k+q_{k-1}-1} h(n) \right) \\
&\le q_{k-1} h(q_k) + \sum_{i=1}^{a_{k+1}-1} q_k  h(iq_k + q_{k-1}) \\
&= q_{k-1} h(q_k) + \sum_{i=1}^{a_{k+1}} q_k  h(iq_k + q_{k-1}) - q_k h (q_{k+1}) \\
&= q_{k-1} h(q_k) + \mu(G_k) - q_k h(q_{k+1}),
\end{align*}
where $q_{-1} =0$, we have
$$\sum_{k=0}^K \sum_{n = q_k}^{q_{k+1} -1} h(n) + q_{K} h(q_{K+1}) \le \sum_{k=0}^K \mu(G_k).$$
From this the claim follows.
\end{proof}

Since
\begin{equation*}
G_k = \bigcup_{i = 0}^{a_{k+1} -1} \left( \bigcup_{q_{k+1} - (i+1)q_{k} < n \le q_{k+1} - iq_{k}} B\left( n\theta, \frac{f (q_{k+1} - iq_{k})}{2}  \right) \right) \subseteq \bigcup_{q_{k-1} < n \le q_{k+1}} B\left( n\theta, \psi(n) \right),
\end{equation*}
we have
\begin{equation*}
\bigcap_{K \ge 0} \bigcup_{k \ge K} G_k \subseteq \bigcap_{N \ge 1} \bigcup_{n \ge N} B \left( n\theta,\psi(n) \right).
\end{equation*}
Therefore, $\sum_{n = 1}^{\infty} f (n) = \infty$ implies that
$$ \mu \left( \bigcap_{N \ge 1} \bigcup_{n \ge N} B \left( n\theta,\psi(n) \right) \right) = 1. $$
This concludes the proof of the divergence part.

\section{The Theorems of Tseng and Kurzweil}\label{Ku-Ts}

In this section, we will give several consequences of Theorem \ref{main-thm}. More precisely, we will show that our result contains three previous theorems. One of them is Kurzweil's theorem mentioned in the introduction. The other two are generalization of Kurzweil's result: the first is due to Tseng \cite{Tseng} and the second is due to Kurzweil himself \cite{Kurzweil}. We start by introducing these two results.

First, we explain Tseng's theorem. Therefore, we need the following notation
\[
\Omega^{(\tau)}:=\{\theta\in{\mathbb R}\ :\ \text{there exists $c>0$ with $\|n\theta\|\geq c/n^{\tau}$ for all $n\geq 1$}\}.
\]
Note that this definition slightly differs from \cite{Tseng}, where $\tau$ was replaced by $\tau-1$. Also, note that $\tau=1$ is by definition the set of badly approximable numbers. Moreover, we let
\[
\Theta^{(\tau)}:=\left\{\theta\in{\mathbb R}\ :\ (\ref{inf-sol})\ \text{holds for all $\psi(n)$ with $\sum_{n\geq 1}\psi(n)^{\tau}=\infty$}\right\}.
\]
Now, we can state Tseng's theorem.
\begin{theorem}[Tseng \cite{Tseng}]
For $\tau\geq 1$, we have,
\[
\Omega^{(\tau)}=\Theta^{(\tau)}.
\]
\end{theorem}
Note that for $\tau=1$ this is Kurzweil's theorem. In his paper \cite{Kurzweil}, Kurzweil himself also gave a generalization of his theorem. We will state this generalization next. Again, we need some notation. First, consider a sequence $\varphi(n)$ with
\begin{align}
&\text{(P1)}\qquad n\varphi(n)\ \text{non-increasing};\label{prop-1}\\
&\text{(P2)}\qquad 0<n^2\varphi(n)\leq 1\ \text{for}\ n\geq 1.\label{prop-2}
\end{align}
For such a sequence, we define
\[
\Omega^{(\varphi)}:=\{\theta\in{\mathbb R}\ :\ \text{there exists $c>0$ with $\|n\theta\|\geq n\varphi(cn)$ for all $n\geq 1$}\}.
\]
Moreover, we consider positive, non-increasing sequences $\psi(n)$ such that there exists an increasing sequence $t_i$ and a non-decreasing function $\delta(n)\geq 1$ which tends to infinity as $n$ tends to infinity with
\begin{equation}\label{series-kurz-3a}
t_{i+1}\geq\frac{1}{t_i\varphi(t_i\delta(t_i))}
\end{equation}
and
\begin{equation}\label{series-kurz-3}
\sum_{i\geq 1}t_i\psi\left(\left\lfloor\frac{1}{t_i\varphi(t_i\delta(t_i))}\right\rfloor\right)=\infty.
\end{equation}
For such sequences, we define
\[
\Xi^{(\varphi)}:=\{\theta\in{\mathbb R}\ :\ (\ref{inf-sol})\ \text{holds for all $\psi(n)$ with the above properties}\}.
\]
Then, Kurzweil's result in \cite{Kurzweil} reads as follows.
\begin{theorem}[Kurzweil \cite{Kurzweil}]\label{Kurz2}
We have,
\[
\Omega^{(\varphi)}=\Xi^{(\varphi)}.
\]
\end{theorem}
Note that $\Omega^{(1/n^{\tau+1})}=\Omega^{(\tau)}$. However, it was shown in \cite{Tseng} that $\Theta^{(\tau)}\ne\Xi^{(1/n^{\tau+1})}$ unless $\tau=1$. More precisely, Tseng showed that for $\tau>1$ neither
\[
\Theta^{(\tau)}\subseteq\Xi^{(1/n^{\tau+1})}\qquad\text{nor}\qquad\Theta^{(\tau)}\supseteq\Xi^{(1/n^{\tau+1})}.
\]

Both of the above result are consequences of Theorem \ref{main-thm} as will be shown next.

\subsection{$\theta\in\Omega^{(\star)}$.} Here, we show that if $\theta\in\Omega^{(\tau)}$ or $\Omega^{(\varphi)}$, then $\theta\in\Theta^{(\tau)}$ or $\Xi^{(\varphi)}$, respectively.

We first consider Tseng's theorem and as a warm-up we begin with the case $\tau=1$ (Kurzweil's theorem).
\begin{lemma}
We have,
\[
\Omega^{(1)}\subseteq\Theta^{(1)}.
\]
\end{lemma}

\begin{proof}
First recall that $\theta\in\Omega^{(1)}$ means that there exists a $c>0$ such that $\|n\theta\|\geq c/n$ for all $n\geq 1$. Thus, for $q_k\leq n<q_{k+1}$
\[
\|q_k\theta\|\geq\frac{c}{q_k}\geq\frac{c}{n}.
\]
This implies that (\ref{nec-suf-cond}) holds provided that
\[
\sum_{n\geq 1}\min\left(\psi(n),\frac{c}{n}\right)=\infty.
\]
By Cauchy's condensation principle, the latter is equivalent to showing that
\[
\sum_{n\geq 0}\min\left(2^n\psi(2^n),c\right)=\infty.
\]
This, in turn follows from $\sum_{n\geq 0}2^n\psi(2^n)=\infty$ which again by Cauchy's condensation principle is equivalent to the assumption.
\end{proof}

We now generalize this to general $\tau$.

\begin{lemma}
For $\tau\geq 1$, we have
\[
\Omega^{(\tau)}\subseteq\Theta^{(\tau)}.
\]
\end{lemma}
\begin{proof}
First, note that with the same argument as in the proof of Lemma \ref{main-thm-conv}, the claim holds when $\psi(q_{k+1}-1)\geq\|q_k\theta\|$ for infinitely many $k$.
Thus, in the sequel, we may assume that $\psi(q_{k}-1)<\|q_{k-1}\theta\|$ for all $k\geq k_0\geq 1$.
Fix $q_{k}\leq n<q_{k+1}$. Then,
\[
\psi(n)^{\tau}\leq\psi(q_k-1)^{\tau}<\|q_{k-1}\theta\|^{\tau}\leq\frac{1}{q_k^{\tau}}\leq\frac{\|q_k\theta\|}{c},
\]
where $c>0$ is such that $\|n\theta\|\geq c/n^{\tau}$ for all $n\geq 1$. Thus, we have
\[
\sum_{n=q_k}^{q_{k+1}-1}\min(\psi(n),\|q_k\theta\|)\geq\min(c,1) \cdot \sum_{n=q_k}^{q_{k+1}-1}\psi(n)^{\tau}.
\]
Summing over $k\geq k_0$ gives
\[
\sum_{k\geq k_0}\sum_{n=q_k}^{q_{k+1}-1}\min(\psi(n),\|q_k\theta\|)\geq\min(c,1) \cdot \sum_{n\geq q_{k_0}}\psi(n)^{\tau}=\infty
\]
which proves the claim also in this case.
\end{proof}

We next show that Theorem \ref{main-thm} also implies one direction of Theorem \ref{Kurz2}.
\begin{lemma}
We have,
\[
\Omega^{(\varphi)}\subseteq\Xi^{(\varphi)}.
\]
\end{lemma}

\begin{proof}
First note that as above, we can assume that $\psi(q_{k}-1)<\|q_{k-1}\theta\|$ for all $k$ large enough.

Now, consider $t_{i-1}\leq n<t_i$. Observe that by $(\ref{prop-2})$ and the assumption of $\delta(n)$, we have
\begin{equation}\label{est-ti}
\frac{1}{t_i\varphi(t_i\delta(t_i))}\geq t_i\delta^{2}(t_i)\geq t_i.
\end{equation}
Thus,
\begin{equation}\label{est-1}
\psi(n)\geq\psi\left(\left\lfloor\frac{1}{t_i\varphi(t_i\delta(t_i))}\right\rfloor\right).
\end{equation}

Next, define $i_s$ such that
\[
q_{i_s-1}<t_i\leq q_{i_s}.
\]
Note that from the assumptions of $\varphi(n)$ and the properties of principal convergents stated at the beginning of the proof of Lemma \ref{main-thm-conv}, we have
\[
q_{i_s-1}\varphi(cq_{i_s-1})\leq\|q_{i_s-1}\theta\|\leq\frac{1}{q_{i_s}}.
\]
From this and (\ref{prop-1}), we obtain that
\[
q_{i_s}\leq\frac{1}{q_{i_s-1}\varphi(cq_{i_s-1})}\leq\frac{c}{t_i\delta(t_i)\varphi(t_i\delta(t_i))}
\leq\frac{1}{t_i\varphi(t_i\delta(t_i))}
\]
for $i$ large enough. Thus, for $q_k\leq n<q_{k+1}$, we have
\[
\|q_k\theta\|\geq\|q_{i_s-1}\theta\|>\psi(q_{i_s}-1)\geq\psi(q_{i_s})\geq
\psi\left(\left\lfloor\frac{1}{t_i\varphi(t_i\delta(t_i))}\right\rfloor\right).
\]

Combining the latter with (\ref{est-1}) yields for $n$ with $t_{i-1}\leq n<t_i$ in the series of (\ref{nec-suf-cond}) the lower bound
\[
(t_i-t_{i-1})\psi\left(\left\lfloor\frac{1}{t_i\varphi(t_i\delta(t_i))}\right\rfloor\right)
\]
for $i$ large enough. Since, from \eqref{series-kurz-3a} and \eqref{est-ti}, we have
\[
t_i\geq t_{i-1}\delta^2(t_{i-1})\geq 2t_{i-1}
\]
for $i$ large enough, we see that a remainder of the series in (\ref{nec-suf-cond}) has a remainder of the series in (\ref{series-kurz-3}) as lower bound which proves the desired result.
\end{proof}

\subsection{$\theta\not\in\Omega^{(\star)}$.} Here, we have to show that there exists a positive, non-increasing sequence $\psi(n)$ satisfying $\sum_{n\geq 1}\psi(n)^{\tau}=\infty$ in case of Tseng's theorem and the condition above Theorem \ref{Kurz2} in case of Theorem \ref{Kurz2} such that (\ref{inf-sol}) does not hold. Such sequences have been constructed already by Tseng and Kurzweil in the proof of their results. One only has to check that these sequences do not satisfy (\ref{nec-suf-cond}). Since the check is the same for all of them, we only give details for Tseng's construction which we recall next.

First, since $\theta\not\in\Omega^{(\tau)}$, there exists a sequence of positive integers $v_{\ell}$ with $v_{\ell+1}\geq 2v_{\ell}$ and
\begin{equation}\label{conv-step-1}
\|v_{\ell}\theta\|\leq\frac{1}{2\ell^{2\tau+2}v_{\ell}^{\tau}}.
\end{equation}
Now, set $u_{\ell}=\lfloor\ell^{2\tau}v_{\ell}^{\tau}\rfloor$ and for $u_{\ell}\leq n<u_{\ell+1}$,
\[
\psi(n)=2^{-1}(\ell+1)^{-2}v_{\ell+1}^{-1}.
\]
Obviously,
\[
\sum_{n=u_{\ell}}^{u_{\ell+1}-1}\psi(n)^{\tau}\geq c
\]
for some constant $c$ and hence $\sum_{n\geq 1}\psi(n)^{\tau}=\infty$. Next, in order to show that (\ref{nec-suf-cond}) does not hold, we set for $q_k\leq n<q_{k+1}$
\[
h(n)=\min\{\psi(n),\|q_k\theta\|\}.
\]
%Then, for $u_{\ell}\leq n<u_{\ell+1}$,
%\[
%h(n)=\min\{2^{-1}(\ell+1)^{-2}v_{\ell+1}^{-1},\|q_k\theta\|\}.
%\]
%(Note that here $k$ is a function of $n$).
Thus,
\begin{align*}
\sum_{n=u_{\ell}}^{u_{\ell+1}-1} h(n) =\sum_{n=u_{\ell}}^{v_{\ell+1}-1} h(n)+ \sum_{n=v_{\ell+1}}^{u_{\ell+1}-1} h(n)
&\leq v_{\ell+1} \psi(u_\ell)  + u_{\ell+1}  \|v_{\ell+1} \theta \| \\
&\leq \frac{v_{\ell+1}}{2(\ell+1)^2v_{\ell+1}} +\frac{u_{\ell+1}}{2(\ell+1)^{2\tau+2}v_{\ell+1}^{\tau}}\leq\frac{1}{(\ell+1)^2},
\end{align*}
where we used (\ref{conv-step-1}) in the above estimate. Summing over $\ell$ shows that (\ref{nec-suf-cond}) does not hold as required.

\section{Khintchine Sequences}\label{Kim}

In this section, we assume that $\psi(n)$ is a Khintchine sequence, i.e., $\psi(n)=1/(n\phi(n))$ with $\phi(n)$ non-decreasing. For this special case, the second author proved in \cite{Kim} the following result.
\begin{theorem}[Kim \cite{Kim}]\label{Khintchine}
Let $\phi(n)$ be a positive, non-decreasing sequence which tends to infinity and $\theta$ be an irrational number with principal convergents $p_k/q_k$. Then, for almost all $s\in{\mathbb R}$,
\[
\|n\theta-s\|<\frac{1}{n\phi(n)}\quad\text{for infinitely many}\ n\in{\mathbb N}
\]
if and only if
\[
\sum_{k=0}^{\infty}\frac{\log(\min(\phi(q_k),q_{k+1}/q_k))}{\phi(q_k)}=\infty.
\]
\end{theorem}

\begin{Rem}
By using the main result of \cite{Kim2} and replacing $\log x$ by ${\rm Log}\ x:=\max\{\log x,0\}$, the assumption that $\phi(n)$ tends to infinity can be dropped (if $\phi(n)$ is bounded, then there are always an infinite number of solutions). Moreover, this can also be obtained by Minkowski's inhomogeneous approximation theorem (e.g., \cite[p.48]{Cassels}) and Cassels' lemma \cite[Lemma 2.1]{Harman}. Note that this situation is also covered by our main result. More precisely, if $\phi(n)\leq c$ for some $c\geq 2$, then we have
\begin{align*}
\sum_{k=0}^{\infty}\left(\sum_{n=q_k}^{q_{k+1}-1}\min\left(\frac{1}{n\phi(n)},\| q_k\theta\|\right)\right)
&\ge \sum_{k=0}^{\infty}\left(\sum_{n=q_{2k}}^{q_{2k+2}-1}\min \left(\frac{1}{cn},\| q_{2k+1} \theta\| \right)\right) \\
%> \sum_{k=0}^{\infty}\left(\sum_{n=q_{2k}}^{q_{2k+2}-1}\min \left( \frac1{cn}, \frac 1{2q_{2k+2}} \right)\right) \\
&\geq\sum_{k=0}^{\infty}\left(\sum_{n=q_{2k}}^{q_{2k+2}-1}\frac 1{cq_{2k+2}} \right)
= \sum_{k=0}^{\infty} \frac{q_{2k+2}-q_{2k}}{cq_{2k+2}}
\geq\sum_{k=0}^{\infty} \frac{1}{2c} = \infty.
\end{align*}
\end{Rem}

Theorem~\ref{Khintchine} is indeed a special case of our Theorem \ref{main-thm} since the following proposition holds.
\begin{proposition}
With the assumptions of the above theorem, we have
\[
\sum_{k=0}^{\infty}\frac{\log(\min(\phi(q_k),q_{k+1}/q_k))}{\phi(q_k)}=\infty.
\]
if and only if
\[
\sum_{k=0}^{\infty}\left(\sum_{n=q_k}^{q_{k+1}-1}\min(\psi(n),\| q_k\theta\|)\right)=\infty,
\]
where $\psi(n)=1/(n\phi(n))$.
\end{proposition}

\begin{proof}
Since we assume that $\phi(n)$ tends to infinity, $\psi(q_{k+1} -1) < \|q_k \theta \|$ for large enough $k$ as before.
For such a large $k$ let $$q^*_k = \min \{ q_k \le n < q_{k+1} : \psi(n) < \| q_k \theta \| \}.$$
Then, we have
\begin{equation*}\label{eq4-1}
\frac{1}{q^*_k} <  \phi (q^*_k) \| q_k \theta \|\leq\frac{\phi (q^*_k)}{q_{k+1}}
\end{equation*}
and if $q^*_k \ge q_k +1$, then
\begin{equation*}\label{eq4-2}
\frac{2}{q^*_k} \ge \frac{1}{q^*_k - 1} \ge \phi(q^*_k -1 ) \| q_k \theta \| \ge \phi(q_k) \| q_k \theta \|\geq \frac{\phi(q_k)}{2q_{k+1}}.
\end{equation*}
Therefore, we have that
$$ \min \left\{ \frac{q_{k+1}}{q_k},  \frac{\phi(q_k)}{4} \right\}\leq \frac{q_{k+1}}{q^*_k} \le \min \left\{ \frac{q_{k+1}}{q_k},  \phi(q^*_k) \right\}$$
and if $q^*_k \ge q_k +1$ and $\phi(q_k) \ge 2e$, then
$$
(q^*_k - q_k) \| q_k \theta \| \le (q^*_k - 1) \| q_k \theta \| \le \frac{1}{\phi(q_k)} \le \frac{\log (\phi(q_k)/2)}{\phi(q_k)} \le \frac{\log ( q_{k+1} \| q_k \theta\| \phi(q_k) )}{\phi(q_k)} \le \frac{\log (q_{k+1}/q_k)}{\phi(q_k)}.
$$

%Then $q^*_k < q_{k+1}$ and

If we consider $\phi$ as a function on $\mathbb R$, then for large $k$ such that $\phi(q_k)\geq 16$, we have
\begin{align*}
\sum_{n=q_k}^{q_{k+1}-1}\min(\psi(n),\| q_k\theta\|)
&\ge \sum_{n=q^*_k}^{q_{k+1}-1} \frac{1}{n\phi(n)} \ge \int_{q^*_k}^{q_{k+1}} \frac{dx}{x \phi(x)} = \int_{\log q^*_k}^{\log q_{k+1}} \frac{dt}{\phi(e^t)} \ge \frac{\log (q_{k+1}/q^*_k)}{\phi(q_{k+1})} \\
& \geq \frac{\log ( \min (\phi(q_k)/4, q_{k+1}/q_k))}{\phi(q_{k+1})}
\geq \frac{\log ( \min (\phi(q_k), q_{k+1}/q_k))}{2 \phi(q_{k+1})}
\end{align*}
and for $\phi(q_k) \ge 2e$
\begin{align*}
\sum_{n=q_k}^{q_{k+1}-1}\min(\psi(n),\| q_k\theta\|)
&= (q^*_{k} - q_k) \| q_k \theta \| + \frac{1}{q^*_k\phi(q^*_k)} + \sum_{n=q^*_k +1}^{q_{k+1}-1} \frac{1}{n\phi(n)} \\
&\le \frac{\log (\min (\phi(q_k), q_{k+1}/q_k))}{\phi(q_k)} + \frac{1}{q_k\phi(q_k)} + \int_{q^*_k}^{q_{k+1}} \frac{dx}{x \phi(x)} \\
%= \frac{2}{\phi(q_k)} + \int_{\log q^*_k}^{\log q_{k+1}} \frac{dt}{\phi(e^t)} \\&
&\le \frac{\log (\min (\phi(q_k), q_{k+1}/q_k))}{\phi(q_k)} + \int_{q^*_{k-1}}^{q_k}\frac{dx}{x \phi(x)}  + \frac{\log ( q_{k+1}/q^*_k)}{\phi(q^*_k)} \\
&\le \frac{\log (\min (\phi(q_k), q_{k+1}/q_k))}{\phi(q_k)} + \frac{\log ( q_{k}/q^*_{k-1})}{\phi(q^*_{k-1})} + \frac{\log (\min (\phi(q^*_k), q_{k+1}/q_k))}{\phi(q^*_k)} \\
&\le \frac{ 2 \log (\min (\phi(q_k), q_{k+1}/q_k))}{\phi(q_k)} + \frac{\log (\min (\phi(q_{k-1}), q_{k}/q_{k-1}))}{\phi(q_{k-1})}.
\end{align*}
Therefore, for some $k_0\geq 1$, we have
$$
\sum_{k>k_0}\frac{\log(\min(\phi(q_k),q_{k+1}/q_k))}{2\phi(q_{k+1})} \le
\sum_{k>k_0}\sum_{n=q_k}^{q_{k+1}-1}\min(\psi(n),\| q_k\theta\|) \le \sum_{k\geq k_0}^{\infty}\frac{3\log(\min(\phi(q_k),q_{k+1}/q_k))}{\phi(q_k)}.
$$

Let $\Lambda = \{ k \ge 1 : \phi(q_{k+1}) \le 2 \phi(q_k) \}.$ Then,
\[
\sum_{k \in \Lambda} \frac{\log(\min(\phi(q_k),q_{k+1}/q_k))}{\phi(q_{k+1})}
\le \sum_{k \in \Lambda} \frac{\log(\min(\phi(q_k),q_{k+1}/q_k))}{\phi(q_{k})}
\le 2 \sum_{k \in \Lambda} \frac{\log(\min(\phi(q_k),q_{k+1}/q_k))}{\phi(q_{k+1})}
\]
and
$$
\sum_{k \in \Lambda^c} \frac{\log(\min(\phi(q_k),q_{k+1}/q_k))}{\phi(q_{k+1})}
\le \sum_{k \in \Lambda^c} \frac{\log(\min(\phi(q_k),q_{k+1}/q_k))}{\phi(q_{k})}
\le \sum_{k \in \Lambda^c} \frac{\log \phi(q_k)}{\phi(q_{k})}< \infty.
$$
Hence, we have
\[
\sum_{k=0}^{\infty}\frac{\log(\min(\phi(q_k),q_{k+1}/q_k))}{\phi(q_{k+1})}=\infty
\ \text{ if and only if } \
\sum_{k=0}^{\infty}\frac{\log(\min(\phi(q_k),q_{k+1}/q_k))}{\phi(q_k)}=\infty,
\]
which completes the proof.
\end{proof}

\section{An Analogue in the Field of Formal Laurent Series}\label{formal-Laur}

In this section, we will briefly discuss an analogue of our Theorem \ref{main-thm} in the field of formal Laurent series. As in the real case, this analogue will imply the analogues of Kurzweil's theorem and its extensions as well as the analogue of Kim's theorem \cite{Kim} which have all been established in the field of formal Laurent series.

We start by recalling the definition of the field of formal Laurent series; for further details see \cite{Fuchs}. First, denote by ${\mathbb F}_q$ the finite field of $q$ elements, where $q$ is a prime power. Moreover, let ${\mathbb F}_q[X]$ be the polynomial ring over ${\mathbb F}_q$ and denote by ${\mathbb F}_q(X)$ its quotient field. The field of formal Laurent series is defined by
\[
{\mathbb F}_q((X^{-1})):=\{f=a_{n}X^{n}+a_{n-1}X^{n-1}+\cdots,\ a_i\in{\mathbb F}_q,\ a_n\ne 0, n\in{\mathbb Z}\}\cup\{0\}
\]
with addition and multiplication defined as for polynomials. We set $\{f\}=a_{-1}X^{-1}+\cdots$ which is called the {\it fractional part} of $f$. Moreover, we define a norm by setting $\vert f\vert=q^{\deg(f)}$, where $\deg(f)$ is the generalized degree function (by definition $\vert 0\vert:=0$). This norm is non-Archimedean. Next, set
\[
{\mathbb L}:=\{f\in{\mathbb F}_q((X^{-1}))\ :\ \vert f\vert<1\}.
\]
Restricting the norm to ${\mathbb L}$ gives a compact topological group. Thus, there exists a unique, translation-invariant probability measure (the Haar measure).

Metric Diophantine approximation is now done in ${\mathbb L}$ equipped with the above measure with integers replaced by elements in ${\mathbb F}_q[X]$ and real number replaced by elements in ${\mathbb L}$. In particular, the inhomogeneous Diophantine approximation problem (\ref{idap}) in this setting becomes
\[
\vert\{Qf\}-g\vert<\frac{1}{q^{l_n}}, \ \deg(Q) = n,
\]
where $f,g\in{\mathbb L}$ and solutions are sought in $Q\in{\mathbb F}_q[X]$. Here, $l_n$ is a non-negative sequence of integers which plays the role of the approximation sequence.

In this setting, our Theorem \ref{main-thm} reads as follows.
\begin{theorem}
Let $l_n$ be a non-decreasing sequence and $f\in{\mathbb L}$ be irrational with principal convergents $P_k/Q_k$.
Then, for almost all $g\in{\mathbb L}$,
\[
\vert\{Qf\}-g\vert <\frac{1}{q^{l_n}},\ \deg(Q)=n\quad \text{for infinitely many $Q\in{\mathbb F}_q[X]$}
\]
if and only if
\[
\sum_{k=0}^{\infty}\left(\sum_{n=n_k}^{n_{k+1}-1}q^{n-\max\{n_{k+1},l_n\}}\right)=\infty,
\]
where $n_k:=\deg(Q_k)$.
\end{theorem}
We remark that under the stronger assumption that $l_n-n$ is non-decreasing this result was already proved in \cite{KiNaNa} (this corresponds to the case of a Khintchine sequence). Also, it was shown in \cite{Lin} that the divergence part holds without the monotonicity assumption. The convergence part (which was conjectured in \cite{Lin}) can be proved with a similar line of reasoning as above. It might be possible to remove also in this case the monotonicity assumption (as is frequently the case for metric Diophantine approximation in the field of formal Laurent series), but we will not pursue this here further.

We conclude by pointing out that the above theorem implies the analogue of Kurzweil's theorem in the field of formal Laurent series which was proved in \cite{ChenFuchs} and \cite{KimNakada}. Moreover, the analogue of Tseng's theorem (proved in \cite{KiTaWaXu}) and the analogue of Theorem \ref{Kurz2} (proved in \cite{Lin}) are also deduced from the above theorem with a similar line of reasoning as in Section \ref{Ku-Ts}. Finally, our result again extends the analogue of Kim's theorem for Khintchine sequences which was obtained in \cite{KiNaNa}. This was already proved in \cite{KiNaNa} and the reader is referred to that paper for details.

\section*{Acknowledgments}
The first author was partially supported by the Ministry of Science and Technology, Taiwan under the grant MOST-103-2115-M-009-007-MY2. The second author was supported by Basic Science Research Program through the National Research Foundation of Korea (NRF) funded by the Ministry of Education, Science and Technology (2012R1A1A2004473). Parts of this work was done while the second author visited the Department of Applied Mathematics, National Chiao Tung University. He thanks the Center of Mathematical Modeling and Scientific Computing (CMMSC) for partial financial support.

\end{document}